\documentclass[12pt,a4paper]{article}
\usepackage[a4paper,margin=2cm]{geometry}

\usepackage{amsmath}
\usepackage{amsthm}
\usepackage{subcaption}
\usepackage{amssymb}
\usepackage{enumitem}
\usepackage{graphicx}
\usepackage{nccmath}
\usepackage{setspace}
\usepackage[colorlinks=true, allcolors=blue]{hyperref}
\usepackage{enumerate}
\usepackage{multirow}
\usepackage{float}
\usepackage[dvipsnames]{xcolor}

\newtheorem{theorem}{Theorem}
\newtheorem{corollary}[theorem]{Corollary}
\newtheorem{lemma}[theorem]{Lemma}

\theoremstyle{definition}

\newtheorem*{claim}{Claim}

\def \mod#1{{\:({\rm mod}\ #1)}}

\def \A{\mathcal{A}}
\def \B{\mathcal{B}}
\def \D{\mathcal{D}}

\def \leq {\leqslant}
\def \geq {\geqslant}
\def \Gc {\overline{G}}
\def \Cf {C_\textup{F}}
\def \Cn {C_\textup{N}}
\def \an {a_\textup{N}}
\def \Vf {V^*_\textup{F}}

\let\oldproofname=\proofname
\renewcommand{\proofname}{\textup{\textbf{\oldproofname}}}

\title{An Evans-style result for block designs}
\author{Ajani De Vas Gunasekara \qquad  Daniel Horsley\\[2mm]
School of Mathematics, Monash University, Victoria 3800, Australia}
\date{}

\begin{document}

\maketitle

\begin{abstract}
For positive integers $n$ and $k$ with $n \geq k$, an \emph{$(n,k,1)$-design} is a pair $(V, \B)$ where $V$ is a set of $n$ \emph{points} and $\B$ is a collection of $k$-subsets of $V$ called \emph{blocks} such that each pair of points occur together in exactly one block. If we weaken this condition to demand only that each pair of points occur together in at most one block, then the resulting object is a \emph{partial $(n,k,1)$-design}. A \emph{completion} of a partial $(n,k,1)$-design $(V,\A)$ is a (complete) $(n,k,1)$-design $(V,\B)$ such that $\A \subseteq \B$. Here, for all sufficiently large $n$, we determine exactly the minimum number of blocks in an uncompletable partial $(n,k,1)$-design. This result is reminiscent of Evans' now-proved conjecture on completions of partial latin squares. We also prove some related results concerning edge decompositions of almost complete graphs into copies of $K_k$.
\end{abstract}


\section{Introduction}

For positive integers $n$, $k$ and $\lambda$ with $n \geq k$, an \emph{$(n,k,\lambda)$-design} is a pair $(V, \B)$ where $V$ is a set of $n$ \emph{points} and $\B$ is a collection of $k$-subsets of $V$ called \emph{blocks} such that each pair of points occur together in exactly $\lambda$ blocks. If we weaken this condition to demand only that each pair of points occur together in at most $\lambda$ blocks, then the resulting object is a \emph{partial $(n,k,\lambda)$-design}. In this paper we are only concerned with $(n,k,1)$-designs and partial $(n,k,1)$-designs. A \emph{completion} of a partial $(n,k,1)$-design $(V,\A)$ is a (complete) $(n,k,1)$-design $(V,\B)$ such that $\A \subseteq \B$. A partial $(n,k,1)$-design is \emph{completable} when it has a completion.  The \textit{leave} of a partial $(n,k,1)$-design $(V,\mathcal{A})$  is the graph $G$ having vertex set $V$ and the edge set $E(G) = \{ xy :  x, y \in V $ such that $ \{x,y\} \nsubseteq A $ for all $ A \in \mathcal{A}  \}$.

Note that an $(n,2,1)$-design exists trivially for each integer $n \geq 2$. It is obvious that if an $(n,k,1)$-design exists then $n(n-1) \equiv 0 \mod{k(k-1)}$ and $n \equiv 1 \mod{(k-1)}$. We call integers $n$ satisfying these restrictions \emph{k-admissible}.  Wilson \cite{Wilson1975} showed that, for each integer $k \geq 3$, there exists an $(n,k,1)$-design for each sufficiently large $k$-admissible value of $n$. Obviously, if a partial $(n,k,1)$-design is completable, then $n$ is $k$-admissible.  Our main result in this paper is to show that, for each sufficiently large $k$-admissible order $n$, all partial $(n,k,1)$-designs with at most $\frac{n-1}{k-1}-k+1$ blocks are completable and that this bound is tight.

\begin{theorem}\label{T:evansBigGenk}
Let $k \geq 3$ be a fixed integer. There is an integer $n_0$ such that for all $k$-admissible integers $n \geq n_0$, any partial $(n,k,1)$-design with at most $\frac{n-1}{k-1}-k+1$ blocks is completable. Furthermore, for all $k$-admissible integers $n \geq (k-1)^2+1$ there is a partial $(n,k,1)$-design with $\frac{n-1}{k-1}-k+2$ blocks that is not completable.
\end{theorem}

The existence of the uncompletable partial designs claimed in Theorem~\ref{T:evansBigGenk} is easily proved (see Lemma~\ref{L:tightnessGenk}(a)). For sufficiently large $n$, Theorem~\ref{T:evansBigGenk} establishes a generalisation of a conjecture of the second author in \cite{Hors2014} that any partial $(n,3,1)$-design having at most $\frac{n-5}{2}$ blocks is completable. Theorem~\ref{T:evansBigGenk} also nicely complements recent results of Nenadov, Sudakov and Wagner \cite{NenadovSudakovWagner2019}. They show that there exist $\epsilon,n_0>0$ such that we can add blocks to any partial $(n,k,1)$-design $(V,\A)$ with $n > n_0$ and $|\A| \leq \epsilon n^2$ to obtain another partial $(n,k,1)$-design whose leave has at most $21k^3\sqrt{|\A|}\,n$ edges. They also show that we can add points and blocks to such a design to obtain a (complete) $(n',k,1)$-design such that $n' \leq n+7k^2\sqrt{|\mathcal{A}|}$.

Theorem~\ref{T:evansBigGenk} is also reminiscent of a well known conjecture of Evans. A partial latin square of order $n$ is an $n \times n$ array in which each cell is either empty or contains an element of $\{1,\ldots,n\}$, and each element of $\{1,\ldots,n\}$ occurs at most once in each row and column. A latin square is a partial latin square with no empty cells. Evans \cite{Evans1960} conjectured that every partial latin square of order $n$ with at most $n-1$ filled cells can be completed to a latin square. This bound is tight because there is a partial latin square of order $n$ with $n$ filled cells that is not completable for each $n \geq 2$. Smetaniuk \cite{Smetaniuk1981} and Anderson and Hilton \cite{AndersonHilton1983} independently proved  Evans' conjecture for all $n$.

There are few completion results available for partial $(n,k,\lambda)$-designs. Colbourn \cite{colbourn paper} has shown that it is $\mathsf{NP}$-complete to decide whether a given partial $(n,3,1)$-design can be completed. In \cite{ColbournCRosa1983} it is observed that partial $(n,3,1)$-designs in which some fixed point is in every block and partial $(n,3,1)$-designs consisting of an odd number of pairwise disjoint blocks are easily seen to be completable. It is then shown that a partial $(n,3,1)$-design is completable if it has two points $x$ and $y$ such that one block contains both $x$ and $y$ and each other block contains either $x$ or $y$.

A \textit{$K_k$-decomposition} of a graph $G$ is a set of copies of $K_k$ in $G$ whose edge sets partition $E(G)$. An $(n,k,1)$-design is equivalent to a $K_k$-decomposition of $K_n$ and a partial $(n,k,1)$-design is equivalent to a $K_k$-decomposition of some subgraph of $K_n$.  Finding a completion of a partial $(n,k,1)$-design is equivalent to finding a $K_k$-decomposition of its leave, and throughout the remainder of the paper we will often view completions in this way. If a graph $G$ has a $K_k$-decomposition, then we must have $|E(G)| \equiv 0 \mod{\binom{k}{2}}$ and $\deg_G(x) \equiv 0 \mod{k-1}$ for each $x \in V(G)$. We call graphs that obey these necessary conditions \emph{$K_k$-divisible}. So Theorem~\ref{T:evansBigGenk} can be rephrased as saying that, for sufficiently large $n$, any graph $G$ on $n$ vertices that is the leave of a partial $(n,k,1)$-design and whose complement has at most $(\frac{n-1}{k-1}-k+1)\binom{k}{2}$ edges, has a $K_k$-decomposition. It is natural to ask whether we can relax the condition that the graph is the leave of a partial design. We prove two subsidiary results which show that this can only be done at the expense of increasing the bound on the number of edges in $G$. Theorem~\ref{T:evansForAdmissibleOrder} considers the case where $G$ need not be a leave but must still have order congruent to $1$ modulo $k-1$, and Theorem~\ref{T:evansForAnyGraph} considers the case where $G$ can be any $K_k$-divisible graph.

\begin{theorem}\label{T:evansForAdmissibleOrder}
Let $k \geq 3$ be a fixed integer. There is an integer $n_0$ such that for all  integers $n \geq n_0$ with $n \equiv 1 \mod{k-1}$, any $K_k$-divisible graph $G$ of order $n$ has a $K_k$-decomposition if

\[|E(G)| > \tbinom{n}{2} - \left(\tfrac{n-1}{k-1}-\ell\right)\tbinom{k}{2} \quad\text{ where }\quad \ell=\tfrac{1}{4}(k^2-k-2).\]
Furthermore, if $k=3$ or $k \equiv 2 \pmod{4}$, then for all $k$-admissible $n \geq \frac{1}{2}k(k-1)^2+1$ there is a $K_k$-divisible graph $G$ of order $n$ such that  $|E(G)| = \tbinom{n}{2} - (\tfrac{n-1}{k-1}-\ell)\tbinom{k}{2}$ and $G$ is not $K_k$-decomposable.
\end{theorem}

\begin{theorem}\label{T:evansForAnyGraph}
Let $k \geq 3$ be a fixed integer. There is an integer $n_0$ such that for all integers $n \geq n_0$, any $K_k$-divisible graph $G$ of order $n$ has a $K_k$-decomposition if
\[|E(G)| >
\begin{cases}
\tbinom{n}{2}-n+\frac{1}{2}(k+1) & \text{if $k \geq 4$}\\
\tbinom{n}{2}-n & \text{if $k=3$.}
\end{cases}
\]
Furthermore, if $k$ divides $s^2-s-1$ for some positive integer $s$, then for $n=s(k-1)+2$ there is a $K_k$-divisible graph $G$ of order $n$ such that  $|E(G)| = \binom{n}{2}-n+\frac{1}{2}(k+1)$ and $G$ is not $K_k$-decomposable. Finally, for each integer $n \geq 12$ with $n \equiv 0 \mod{6}$, there is a $K_3$-divisible graph $G$ of order $n$ such that $|E(G)| = \binom{n}{2}-n$ and $G$ is not $K_3$-decomposable.
\end{theorem}

The case division in Theorem~\ref{T:evansForAnyGraph} is due to the fact that we go to a little extra effort to obtain a tight bound for the special case $k=3$. Note that there are infinitely many values of $k$, all of them odd, such that $k$ divides $s^2-s-1$ for some positive integer $s$. From Theorems~\ref{T:evansForAdmissibleOrder} and~\ref{T:evansForAnyGraph} it is not too difficult to determine the maximum number of edges in a graph of order $n$ that is $K_3$-divisible but not $K_3$-decomposable for all sufficiently large $n$.

\begin{corollary}\label{C:k=3}
There is an integer $n_0$ such that for all integers $n \geq n_0$, any $K_3$-divisible graph $G$ of order $n$ has a $K_3$-decomposition if $|E(G)| > \tbinom{n}{2} - e(n)$, where
\[ e(n)=
\begin{cases}
\frac{1}{2}(3n-9) & \text{if $n \equiv 1, 3 \mod{6}$}\\
\frac{1}{2}(3n-7) & \text{if $n \equiv 5 \mod{6}$}\\
n+2 & \text{if $n \equiv 2,4 \mod{6}$}\\
n & \text{if $n \equiv 0 \mod{6}$.}
\end{cases}\]
Furthermore, for each $n \geq 7$ there is a $K_3$-divisible graph $G$ of order $n$ such that  $|E(G)| = \tbinom{n}{2} - e(n)$ and $G$ is not $K_3$-decomposable.
\end{corollary}

Very recently, Gruslys and Letzter \cite{GruLet} have proved that any graph of order $n \geq 7$ with strictly more than $\binom{n}{2}-(n-3)$ edges has a \emph{fractional} $K_3$-decomposition. This makes an interesting comparison with Theorem~\ref{T:evansForAnyGraph} and Corollary~\ref{C:k=3}. Considering complements, Theorems~\ref{T:evansForAdmissibleOrder} and~\ref{T:evansForAnyGraph} can be thought of as concerning which graphs are or are not the leaves of partial $(n,k,1)$-designs. This question has received some attention: see \cite[Chapter 9]{ColRosBook}, \cite[\S40.4]{StiWeiYin} and the references therein, for example. Perhaps closest to our concerns here, the possible sizes of \emph{triangle-free} graphs whose complements are $K_3$-divisible but not $K_3$-decomposable are considered in \cite{StiWal}. Our results here improve the lower bounds in that paper.

\section{Preliminaries}\label{S:prelim}

For a family $\A$ of subsets of a set $V$ and an element $x \in V$, we let $\A_x=\{A \in \A: x\in A\}$. For a set $A$ of vertices we use $K_A$ to denote the complete graph with vertex set $A$. For a graph $G$ and a subset $S$ of $V(G)$, we denote by $G[S]$ the subgraph of $G$ induced by $S$. We also denote the minimum and maximum degree of $G$ by $\delta(G)$ and $\Delta(G)$ and the complement of $G$ by $\overline{G}$. For graphs $G$ and $H$ we denote by $G \cup H$ the graph with vertex set $V(G) \cup V(H)$ and edge set $E(G) \cup E(H)$ and denote by $G-H$ the graph with vertex set $V(G)$ and edge set $E(G) \setminus E(H)$. For a positive integer $r$, a \emph{$K_r$-factor} of a graph $G$ is a set of copies of $K_r$ in $G$ whose vertex sets partition $V(G)$. For vertices $x$ and $y$ of a graph $G$, we use $N_G(x,y)$ to denote the mutual neighbourhood $N_G(x) \cap N_G(y)$ of $x$ and $y$. In Lemma~\ref{L:tightnessGenk}(a), (b) and (c) below, we  establish the tightness claims in Theorems~\ref{T:evansBigGenk} and \ref{T:evansForAdmissibleOrder} and in the $k\geq4$ case of Theorem~\ref{T:evansForAnyGraph} respectively.

\begin{lemma}\label{L:tightnessGenk} Let $k \geq 3$ be an integer.
\begin{itemize}
    \item[\textup{(a)}]
For all $k$-admissible integers $n \geq (k-1)^2+1$ there is a partial $(n,k,1)$-design with $\frac{n-1}{k-1}-k+2$ blocks that is not completable.
    \item[\textup{(b)}]
If $k=3$ or $k \equiv 2 \pmod{4}$ then, for all $k$-admissible integers $n \geq \frac{1}{2}k(k-1)^2+1$, there is a $K_k$-divisible graph $G$ of order $n$ such that
\[|E(\overline{G})| = \left(\tfrac{n-1}{k-1}-\tfrac{1}{4}(k^2-k-2)\right)\tbinom{k}{2}\]
and $G$ is not $K_k$-decomposable.
    \item[\textup{(c)}]
If $k$ divides $s^2-s-1$ for some positive integer $s$ then, for $n=s(k-1)+2$, there is a $K_k$-divisible graph $G$ of order $n$ with $|E(\overline{G})|=n-\frac{1}{2}(k+1)$ that is not $K_k$-decomposable.
\end{itemize}
\end{lemma}

\begin{proof}
We first prove (a). Let $(V, \A)$ be a partial $(n,k,1)$-design with $|\A|= \frac{n-1}{k-1}-k+2$ such that $\frac{n-1}{k-1}-k+1$ blocks each contain some fixed point $z \in  V$ and the remaining block, say $A_0$, is disjoint from every other block in $\A$. So $|\A_{z}| = \frac{n-1}{k-1}-k+1$. Suppose for a contradiction that $(V,\B)$ is a completion of $(V, \A)$. In $(V,\B)$ each point lies in exactly $\frac{n-1}{k-1}$ blocks. Thus $|\B_z \setminus \A_z|=k-1$. But $|\B_z \setminus \A_z| \geq k$ because each pair in $\{\{x,z\}:x \in A_0\}$ must occur in a different block. This is a contradiction.

We now prove (b). If $k=3$ then the leave of the partial $(n,k,1)$-design defined in (a) has the required properties, so we may assume that $k \equiv 2 \mod{4}$. Let $V$ be a set of $n$ vertices and let $z \in V$. Let $t=\frac{n-1}{k-1}-\frac{k}{2}(k-1)$ and let $A_1,\ldots,A_t$ be $k$-subsets of $V$ such that $A_i \cap A_j=\{z\}$ for all distinct $i,j \in \{1,\ldots,t\}$. Let $A_0$ be a $(\frac{k}{2}(k-1)+1)$-subset of $V$ such that $A_0$ is disjoint from $A_i$ for all $i \in \{1,\ldots,t\}$. Take $G$ to be the graph $K_V-\bigcup_{i=0}^tK_{A_i}$ and note
\[|E(\overline{G})| = t\tbinom{k}{2}+\tbinom{k(k-1)/2+1}{2} = \left(\tfrac{n-1}{k-1}-\tfrac{1}{4}(k^2-k-2)\right)\tbinom{k}{2}.\]
Furthermore, $\deg_{\Gc}(x) \equiv 0 \mod{k-1}$ for each $x \in V$ and hence, using the fact that $K_V$ is $K_k$-divisible since $n$ is $k$-admissible, we have that $G$ is $K_k$-divisible.
Now suppose for a contradiction there is a $K_k$-decomposition $\D$ of $G$. We have $\deg_{G}(z)=n-1-t(k-1)=\frac{k}{2}(k-1)^2$, so $z$ is a vertex of exactly $\frac{k}{2}(k-1)$ copies of $K_k$ in $\D$. But $z$ must be a vertex of at least $|A_0|=\frac{k}{2}(k-1)+1$ copies of $K_k$ in $\D$ because each edge in $\{xz:x \in A_0\}$ must occur in a different copy of $K_k$. This is a contradiction.

Finally, we prove (c). Let $V$ be a set of $n$ vertices, where $n=s(k-1)+2$ for some positive integer $s$ with $s^2 -s-1 \equiv 0 \mod{k}$, and let $z \in V$. Observe that $k$ is odd since $s^2 -s-1$ is odd. Let $G$ be a graph on vertex set $V$ such that $\Gc$ is the vertex-disjoint union of a star with $n-k$ edges centred at $z$ and a perfect matching on the remaining $k-1$ vertices. Note that $|E(G)| = \binom{n}{2}-n+\frac{1}{2}(k+1)$ and hence that $|E(G)| \equiv 0 \mod{\binom{k}{2}}$ because $n = s(k-1)+2$ and $s^2(k-1)^2 +s(k-1)+(k-1) \equiv 0 \mod{k(k-1)}$. Furthermore, $\deg_G(z)=k-1$ and $\deg_G(x)=n-2=s(k-1)$ for all $x \in V \setminus \{z\}$ and hence $G$ is $K_k$-divisible.  Let $U = N_G(z)$ and note that any $K_k$-decomposition of $G$ must include a copy of $K_k$ with vertex set $\{z\} \cup U$. But this is impossible because $\overline{G}[U]$ is a perfect matching on $k-1$ vertices.
\end{proof}

Note that the construction from the proof of Lemma~\ref{L:tightnessGenk}(b) cannot be converted into a counterexample to Theorem~\ref{T:evansBigGenk} because, by Fisher's inequality \cite{Fis}, $K_{k(k-1)/2+1}$ is not $K_k$-decomposable. Also observe that Theorem~\ref{T:evansBigGenk} is tight for almost all feasible values of $k$ and $n$, while Theorems~\ref{T:evansForAdmissibleOrder} and \ref{T:evansForAnyGraph} are tight only for some values of $k$. So there remains the possibility that the bounds in Theorems~\ref{T:evansForAdmissibleOrder} and \ref{T:evansForAnyGraph} can be improved for particular values of $k$.

We also require some examples of graphs that are not $K_3$-divisible to establish the tightness claims in the $k=3$ case of Theorem~\ref{T:evansForAnyGraph} and in Corollary~\ref{C:k=3}. Note that we have already shown that Corollary~\ref{C:k=3} is tight for $n \equiv 1,3 \mod{6}$ in Lemma~\ref{L:tightnessGenk}(b).

\begin{lemma}\label{L:k=3tightness}\phantom{a}\vspace*{-1.5mm}
\begin{itemize}
    \item[\textup{(a)}]
For each integer $n \geq 12$ with $n \equiv 0 \mod{6}$, there is a $K_3$-divisible graph $G$ of order $n$ with $|E(\overline{G})|=n$ that is not $K_3$-decomposable.
    \item[\textup{(b)}]
For each integer $n \geq 11$ such that $n \equiv 5 \mod{6}$ there is a $K_3$-divisible graph $G$ of order $n$ with $|E(\overline{G})|=\frac{1}{2}(3n-7)$  that is not $K_3$-decomposable.
    \item[\textup{(c)}]
For each integer $n \geq 8$ such that $n \equiv 2,4 \mod{6}$ there is a $K_3$-divisible graph $G$ of order $n$ with $|E(\overline{G})|=n+2$ that is not $K_3$-decomposable.
\end{itemize}
\end{lemma}

\begin{proof}
We first prove (a). Let $V$ be a set of $n$ vertices, where $n \geq 12$ and $n \equiv 0 \mod{6}$, and let $z \in V$. Let $G$ be a graph on vertex set $V$ such that $\Gc$ is the vertex-disjoint union of a star with $n-7$ edges centred at $z$, a copy of $K_4$ with some vertex set $A$, and a copy of $K_2$. Clearly $|E(\Gc)|=n$ and $G$ is $K_3$-divisible.
A $K_3$-decomposition of $G$ must contain exactly three copies of $K_3$ that have $z$ as one of their vertices, but each of the four edges between $z$ and a vertex in $A$ must occur in a different copy of $K_3$. So $G$ has no $K_3$-decomposition.

We now prove (b). Let $V$ be a set of $n$ vertices, where $n \geq 11$ and $n \equiv 5 \mod{6}$, and let $z \in V$. Let $G$ be a graph on vertex set $V$ such that $\Gc$ is the union of $\frac{1}{2}(n-9)$ edge-disjoint copies of $K_3$ whose vertex sets pairwise have intersection $\{z\}$, a copy of $K_5$ with some vertex set $A$ that is disjoint from the vertex set of each copy of $K_3$, and three isolated vertices. It is easy to check that $|E(\overline{G})|=\frac{1}{2}(3n-7)$ and $G$ is $K_3$-divisible. A $K_3$-decomposition of $G$ must contain exactly four copies of $K_3$ that have $z$ as one of their vertices, but each of the five edges between $z$ and a vertex in $A$ must occur in a different copy of $K_3$. So $G$ has no $K_3$-decomposition.

Finally we prove (c). Let $V$ be a set of $n$ vertices, where $n \geq 8$ and $n \equiv 2,4 \mod{6}$. Let $G$ be a graph on vertex set $V$ such that $\Gc$ is the union of a star with $n-3$ edges centred at $z$ and the graph with edge set $\{ux,uy,vx,vy,xy\}$, where $u$ and $v$ are distinct leaf vertices of the star and $x$ and $y$ are the two vertices of $V$ not in the star. It is easy to check that $|E(\overline{G})|=n+2$ and $G$ is $K_3$-divisible. A $K_3$-decomposition of $G$ must contain a copy of $K_3$ with vertex set $\{x,y,z\}$ but this is impossible since $xy \in E(\Gc)$.
\end{proof}

The rest of the paper is devoted to proving the first parts of the theorems and Corollary~\ref{C:k=3}. Our approach is based on the fact that $K_k$-divisible graphs with large order and high minimum degree are known to be $K_k$-decomposable. For each integer $k \geq 3$, $\delta_{K_k}$ is defined to be the infimum of all positive real numbers $\delta$ that satisfy the following: there is a positive integer $n_0$ such that every $K_k$-divisible graph of order $n>n_0$ and minimum degree at least $\delta n$ has a $K_k$-decomposition. Delcourt and Postle \cite{DelPos} have shown that $\delta_{K_3} \leq 0.82733$ and Montgomery \cite{Mon} has shown that $\delta_{K_k} \leq 1-\frac{1}{100k}$ for each $k \geq 4$. Both of these results rely on the work of Glock, K\"{u}hn, Lo, Montgomery and Osthus in \cite{GlockKuhnLoMontgomeryOsthus2019}. For our purposes here, it is enough to know that  $\delta_{K_k} < 1$ for each $k \geq 3$. Often, simply applying this fact to an almost complete graph will show it to be $K_k$-decomposable. However, this approach will not work if the graph contains vertices of low degree. In these situations we follow \cite{NenadovSudakovWagner2019} in deleting copies of $K_k$ from the graph until the vertices that began with low degree become isolated. We can then remove the isolated vertices and apply the fact that $\delta_{K_k}<1$ to the resulting graph to show that the original graph is $K_k$-decomposable. We will make use of the following well known theorems of Tur\'an and of Hajnal and Szemer\'{e}di.

\begin{theorem}[\cite{Turan1941}]\label{T:Turan}
Let $r \geq 2$ be an integer. If a graph $H$ has more than $\frac{r-2}{2r-2}|V(H)|^2$ edges, then it contains a copy of $K_{r}$.
\end{theorem}

\begin{theorem}[\cite{HajnalSzemeredi1970}]\label{T:HS}
Let $r$ be a positive integer. If a graph $H$ has $|V(H)| \equiv 0 \mod{r}$ and $\delta(H) \geq \frac{r - 1}{r}|V(H)|$, then it contains a $K_{r}$-factor.
\end{theorem}

The following simple inductive argument encapsulates the basics of our approach. Given a graph $G$ on an indexed vertex set $\{z_1,\ldots,z_s\}$ and two edges $z_iz_j$ and $z_{i'}z_{j'}$ of $G$ where $i<j$ and $i'<j'$, we say that $z_iz_j$ \emph{lexicographically precedes} $z_{i'}z_{j'}$ if either $i < i'$ or $i=i'$ and $j<j'$. Recall that $N_G(x,y)$ is the mutual neighbourhood $N_G(x) \cap N_G(y)$ of $x$ and $y$.

\begin{lemma} \label{L:K_k decomposition1}
Let $k \geq 3$ be a fixed integer and let $\gamma <1-\delta_{K_k}$ be a positive constant. For all sufficiently large integers $n$ the following holds.
Let $G$ be a $K_k$-divisible graph of order $n$, let $S=\{z_1,\ldots,z_s\}$ be an indexed subset of $V(G)$, and suppose that
\begin{itemize}
    \item[\textup{(i)}]
$|N_G(x) \setminus S| \geq (1-\gamma)n+(k-2)|N_G(x) \cap S|$ for each $x \in V(G) \setminus S$;
    \item[\textup{(ii)}]
either $N_G(z)=\emptyset$ or $|N_G(z) \setminus S| > (k-1)\gamma n +(k-2)|N_G(z) \cap S|$ for each $z \in S$;
    \item[\textup{(iii)}]
for any $i,j \in \{1,\ldots,s\}$ such that $i<j$ and $z_iz_j \in E(G)$ we have
\[|N_G(z_i,z_j) \setminus S| > (k-3)\gamma n+(k-2)\ell_G(z_{i}z_{j})\]
where $\ell_G(z_{i}z_{j})=|N_G(z_i) \cap \{z_1,\ldots,z_{j-1}\}|+|N_G(z_j) \cap \{z_1,\ldots,z_{i-1}\}|$ is the number of edges of $G[S]$ that are adjacent to $z_iz_j$ and lexicographically precede it.
\end{itemize}
Then $G$ has a $K_k$-decomposition.
\end{lemma}

\begin{proof}
We prove the result by induction on the quantity $\sigma(G)=\sum_{z \in S}\deg_G(z)$. Let $s=|S|$. If $\sigma(G)=0$, then the vertices in $S$ are isolated and $\deg_G(x) \geq (1-\gamma)n \geq (1-\gamma)(n-s)$ for each $x \in V(G) \setminus S$ by (i). So the graph obtained from $G$ by deleting the vertices in $S$ is $K_k$-decomposable by the definition of $\delta_{K_k}$ since $\gamma <1-\delta_{K_k}$, and thus the result follows. So we may assume that $\sigma(G) > 0$.

We consider two cases according to whether $G[S]$ is empty. In each case we form a new graph $G'$ from $G$ by removing the edges of some number of copies of $K_k$ in $G$ and then complete the proof by showing that $G'$ satisfies the inductive hypotheses. Note that $G'$ will be $K_k$-divisible because $G$ is $K_k$-divisible. In what follows it will be useful to observe that (i) implies that the vertex $x$ is nonadjacent to at most $\gamma n$ vertices in $G$ (including itself) for each $x \in V(G) \setminus S$. \smallskip

\textbf{Case 1:} Suppose that $G[S]$ is not empty. Let $z_iz_j$, where $i <j$, be the lexicographically first edge in $G[S]$. Let $H$ be the subgraph of $G$ induced by $N_G(z_i,z_j) \setminus S$. Then $|V(H)| > (k-3)\gamma n$ by (iii). We claim that there is a subset $X$ of $V(H)$ such that $H[X]$ is a copy of $K_{k-2}$. If $k=3$, this is immediate because $|V(H)| > 0$. If $k \geq 4$, then $\deg_{H}(x) \geq |V(H)|-\gamma n > \frac{k-4}{k-3}|V(H)|$ for each $x \in V(H)$ where the first inequality follows by (i) and the second from $|V(H)| > (k-3)\gamma n$. So it follows from Theorem~\ref{T:Turan} that such an $X$ exists. Let $G'=G-K_B$ where $B=X\cup \{z_i,z_j\}$. Note that $\sigma(G') < \sigma(G)$, so it suffices to show that $G'$ satisfies (i), (ii) and (iii).

Observe that $|N_{G'}(x) \setminus S|=|N_{G}(x) \setminus S|-(k-3)$ and $|N_{G'}(x) \cap S|=|N_G(x) \cap S|-2$ for each $x \in X$, and $N_{G'}(x) =N_G(x)$ for each $x \in V \setminus (S \cup X)$. Thus $G'$ satisfies (i) because $G$ satisfies (i). Also, $|N_{G'}(z) \setminus S|=|N_{G}(z) \setminus S|-(k-2)$ and $|N_{G'}(z) \cap S|=|N_G(z) \cap S|-1$ for each $z \in\{z_i,z_j\}$, and $N_{G'}(z) =N_{G}(z)$ for each $z \in S \setminus \{z_i,z_j\}$. Thus $G'$ satisfies (ii) because $G$ satisfies (ii). If $G'[S]$ is empty, then $G'$ satisfies (iii) trivially. Otherwise, let $z_{i'}z_{j'}$ be an arbitrary edge in $G'[S]$ where $i'<j'$. If $\{i',j'\} \cap \{i,j\}=\emptyset$, then $N_{G'}(z_{i'},z_{j'})  \setminus S = N_{G}(z_{i'},z_{j'}) \setminus S$ and $\ell_{G'}(z_{i'}z_{j'})=\ell_{G}(z_{i'}z_{j'})$. Otherwise either $i'=i$ and $j'>j$ or $i'=j$ by our definition of $z_iz_j$. Then $|N_{G'}(z_{i'},z_{j'}) \setminus S| \geq |N_{G}(z_{i'},z_{j'}) \setminus S| - (k-2)$ and $\ell_{G'}(z_{i'}z_{j'})=\ell_{G}(z_{i'}z_{j'})-1$. Thus $G'$ satisfies (iii) because $G$ satisfies (iii).

\textbf{Case 2:} Suppose that $G[S]$ is empty. Because $\sigma(G)>0$, there is an $i \in \{1,\ldots,s\}$ such that $N_G(z_i) \neq \emptyset$. Let $H$ be the subgraph of $G$ induced by $N_G(z_i)$. By (ii), $|V(H)| > (k-1)\gamma n$ and, because $G$ is $K_k$-divisible, $|V(H)| = t(k-1)$ for some integer $t$. By (i), for each $x \in V(H)$, we have $\deg_{H}(x) \geq |V(H)|-\gamma n > \frac{k-2}{k-1}|V(H)|$. So Theorem~\ref{T:HS} implies that there is a partition $\{X_1,\ldots,X_t\}$ of $V(H)$ such that $H[X_j]$ is a copy of $K_{k-1}$ for each $j \in \{1,\ldots,t\}$. Let $G'=G-\bigcup_{j=1}^tK_{B_j}$ where $B_j=X_j \cup \{z_i\}$ for each $j \in \{1,\ldots,t\}$.

Observe that $|N_{G'}(x) \setminus S|=|N_{G}(x) \setminus S|-(k-2)$ and $|N_{G'}(x) \cap S|=|N_G(x) \cap S|-1$ for each $x \in V(H)$, and $N_{G'}(x)=N_G(x)$ for each $x \in V \setminus (S \cup V(H))$. Thus $G'$ satisfies (i) because $G$ satisfies (i). Also, $N_{G'}(z_i) = \emptyset$ and $N_{G'}(z) = N_{G}(z)$ for each $z \in S \setminus \{z_i\}$. Thus $G'$ satisfies (ii) because $G$ satisfies (ii). Furthermore,  $G'[S]$ is empty and so $G'$ satisfies (iii) trivially.
\end{proof}

Note that $|N_G(x) \cap S|$ in conditions (i) and (ii) of Lemma \ref{L:K_k decomposition1} is at most $s$, and $\ell_G(z_{i}z_{j})$ in condition (iii) is less than $2s$. This will be useful to remember when we apply Lemma~\ref{L:K_k decomposition1} below. We only require Lemma \ref{L:K_k decomposition1} in order to prove our next result,  Lemma~\ref{L:mutualNeighbourhoodDecomp}, which may be of some independent interest. It shows that we can guarantee a $K_k$-divisible graph with a positive proportion of non-edges has a $K_k$-decomposition if we further require that each edge is in sufficiently many triangles.

\begin{lemma}\label{L:mutualNeighbourhoodDecomp}
Let $k \geq 3$ be a fixed integer, and let $\gamma < 1-\delta_{K_k}$ be a positive constant. For any sufficiently large integer $n$, a $K_k$-divisible graph $G$ of order $n$ is $K_k$-decomposable if $|E(G)| \geq (1-\frac{1}{4k}\gamma^2)\binom{n}{2}$ and $|N_G(x,y)|>k \gamma n$ for each $xy \in E(G)$.
\end{lemma}

\begin{proof}
Let $G$ be a $K_k$-divisible graph of order $n$ with $|E(G)| \geq (1-\frac{1}{4k}\gamma^2)\binom{n}{2}$ and $|N_G(x,y)|>k \gamma n$ for each $xy \in E(G)$. Note that $|E(\overline{G})| \leq \frac{1}{4k}\gamma^2\binom{n}{2}$. Let $S=\{x \in V(G): \deg_{\overline{G}} (x) \geq \frac{1}{2}\gamma n\}$ and $|S| = s$. So we have $\frac{1}{2}\gamma n s \leq 2|E(\overline{G})| \leq \frac{1}{2k}\gamma^2\binom{n}{2}$, and hence $s < \frac{1}{2k} \gamma n$. It suffices to show that $G$ and $S$ satisfy conditions (i), (ii) and (iii) of Lemma~\ref{L:K_k decomposition1}.

\textbf{(i)} Consider any vertex $x \in V(G) \setminus S$. We have $\deg_{G}(x) > (1-\frac{1}{2}\gamma)n - 1$ by the definition of $S$. Therefore, $|N_{G}(x) \setminus S| > (1-\frac{1}{2}\gamma)n - 1 -s > (1-\frac{k+1}{2k}\gamma)n - 1$. Thus, condition (i) of Lemma~\ref{L:K_k decomposition1} holds, noting that $(k-2)|N_{G}(x) \cap S| \leq (k-2)s < \frac{k-2}{2k}\gamma n$ in that condition.

\textbf{(ii)} Consider any vertex $x \in S$. If $N_G(x)=\emptyset$, then (ii) is satisfied for $x$. Otherwise, for any vertex $y \in V(G)$  such that $xy \in E(G)$, we have $|N_G(x,y)|> k\gamma n$ by our hypotheses,
and hence
\begin{equation}\label{E:mutualNeighbourhoodDecomp1}
|N_{G}(x) \setminus S| \geq |N_G(x,y) \setminus S|> k\gamma n-s>(k-\tfrac{1}{2k})\gamma n.
\end{equation}
Thus condition (ii) of Lemma~\ref{L:K_k decomposition1} holds, noting that $(k-2)|N_{G}(x) \cap S| \leq (k-2)s < \frac{k-2}{2k}\gamma n$ in that condition.

\textbf{(iii)} Consider any edge $xy \in E(G[S])$. By \eqref{E:mutualNeighbourhoodDecomp1}, we have $|N_G(x,y) \setminus S|>(k-\frac{1}{2k})\gamma n$. Thus, condition (iii) of Lemma~\ref{L:K_k decomposition1} holds, noting that $(k-2)\ell_{G}(xy) < 2(k-2)s < \frac{k-2}{k} \gamma n$ in that condition.
\end{proof}

\section{Proof of Theorem~\ref{T:evansBigGenk}}

Suppose that $(V,\A)$ is a partial $(n,k,1)$-design with $|\A|=\frac{n-1}{k-1}-k+1$ and that $G$ is its leave. One important situation in which we cannot complete $(V,\A)$ by applying Lemma~\ref{L:mutualNeighbourhoodDecomp} to $G$ is when there is a point $z \in V$ which is in nearly every block in $\A$ (since then edges of $G$ incident with $z$ will not be in enough triangles). In this case, completing $(V,\A)$ will necessarily involve finding a $K_{k-1}$-factor in $G[N_G(z)]$. Lemma~\ref{L:colouring} below allows us to accomplish this task. It is simpler and more natural to consider the complement and state the result in terms of a colouring of a union of cliques. A \emph{proper colouring of a graph $H$ with colour set $C$} is an assignment $\varphi:V(H) \rightarrow C$ of colours from $C$ to the vertices of $H$ such that adjacent vertices receive different colours. The \emph{colour class} of a colour $c \in C$ under $\varphi$ is the set $\varphi^{-1}(c)$ of all vertices to which $\varphi$ assigns colour $c$.

The basic strategy in the proof of Lemma~\ref{L:colouring} is the commonly-used one of colouring vertices greedily according to a degeneracy ordering. A \emph{degeneracy ordering} $v_1,\ldots,v_n$ of the vertices of a graph $H$ is one for which $v_i$ is a vertex of minimum degree in $H[\{v_1,\ldots,v_i\}]$ for each $i \in \{1,\ldots,n\}$. Such an ordering is easily obtained by choosing a vertex of minimum degree in a graph, deleting it and placing it last in the ordering, and repeating this procedure recursively. Sometimes our greedy strategy will get stuck, however, and in these cases we will be forced to recolour an already-coloured vertex.

\begin{lemma}\label{L:colouring}
Let $k$ and $a$ be  integers with $k \geq 3$ and $a \geq k-1$, let $V$ be a set of $a(k-1)$ vertices, and let $\A$ be a set of subsets of $V$ such that $|\A| \leq a-k+1$, $|A| \leq k$ for all $A \in \A$ and $|A \cap A'| \leq 1$ for all distinct $A,A' \in \A$. The graph $H$ with vertex set $V$ and edge set $\bigcup_{A \in \A}E(K_{A})$ has a proper colouring with $a$ colours such that each colour class has order $k-1$.
\end{lemma}

\begin{proof}
Let $C$ be a set of $a$ colours. For the duration of this proof we call a proper colouring \emph{legal} if its colour set is (a subset of) $C$ and each of its colour classes has order at most $k-1$. Let $v_1,\ldots,v_{a(k-1)}$ be a degeneracy ordering of the vertices in $V$. Let $V_i=\{v_1,\ldots,v_i\}$ and $H_i=H[V_i]$ for each $i \in \{1,\ldots,a(k-1)\}$. Clearly $H_{a}$ has a legal colouring as we may colour each vertex with a different colour. We assume that there is a legal colouring $\varphi_{j-1}$ of $H_{j-1}$ for some $j \in \{a+1,\ldots,a(k-1)\}$ and proceed to show that we can find a legal colouring of $\varphi_j$ of $H_{j}$. Extending $\varphi_{j-1}$ by assigning $v_j$ a new colour $c$ might fail to result in a legal colouring for two reasons: either $c$ may already be assigned by $\varphi_{j-1}$ to $k-1$ vertices or $c$ may be assigned by $\varphi_{j-1}$ to a vertex adjacent in $H_j$ to $v_j$.  Accordingly, let $\Cf=\{c \in C: |\varphi_{j-1}^{-1}(c)|=k-1\}$, let $\Cn$ be the set of colours in $C$ that are assigned by $\varphi_{j-1}$ to vertices adjacent in $H_j$ to $v_j$, and let $\an=|\Cn|$. We think of colours in $\Cf$ as ``full'' and those in $\Cn$ as ``neighbouring''.

If $C \setminus (\Cf \cup \Cn)$ is nonempty, then we can extend $\varphi_{j-1}$ to a legal colouring $\varphi_j$ of $H_j$ by assigning any colour in $C \setminus (\Cf \cup \Cn)$ to $v_j$. So we may assume that $\Cf \cup \Cn=C$. Since $j-1$, the number of vertices already coloured, is less than $a(k-1)$, it follows from the definition of $\Cf$ that $|\Cf| < a$ and hence that $\Cn \setminus \Cf \neq \emptyset$ and $\an \geq 1$. Let $c'$ be a colour in $\Cn \setminus \Cf$ and let $V'=\varphi^{-1}_{j-1}(c')$. Let $\Vf=\bigcup_{c \in \Cf \setminus \Cn}\varphi^{-1}_{j-1}(c)$ be the set of vertices already assigned a colour in $\Cf \setminus \Cn$. We aim to proceed by colouring $v_j$ with a colour in $\Cf \setminus \Cn$ but also recolouring a vertex of that colour with $c'$. We will be able to do this if the following claim holds.
\begin{claim}
There is a vertex in $\Vf$ that is not adjacent in $H_j$ to any vertex in $V'$.
\end{claim}
\noindent If this claim is true, we can let $u$ be such a vertex in $\Vf$ and let $\varphi_j$ be the colouring of $H_j$ such that $\varphi_j(v_j)=\varphi_{j-1}(u)$, $\varphi_j(u)=c'$, and $\varphi_j(x)=\varphi_{j-1}(x)$ for each $x \in V_{j-1} \setminus \{u\}$. Since $\varphi_{j-1}(u) \notin \Cn$ and $u$ is not adjacent in $H_j$ to any vertex in $V'$, it can be seen that $\varphi_j$ is a proper colouring and since $c' \notin \Cf$ it can be seen that $\varphi_j$ is a legal colouring. So it suffices to prove our claim.

\noindent \textbf{Proof of claim.} Suppose for a contradiction that each vertex in $\Vf$ is adjacent in $H_j$ to some vertex in $V'$. Observe that $V'$ and $\Vf$ are disjoint and that
\begin{equation}\label{E:verticesLowerBounds}
|V'| \geq 1, \qquad |\Vf|=(k-1)(a-\an) \qquad \text{and} \qquad |V_j \setminus (V' \cup \Vf)| \geq \an
\end{equation}
where the second of these follows because each of the $a-\an$ colours in $\Cf \setminus \Cn$ is assigned by $\varphi_{j-1}$ to exactly $k-1$ vertices in $V_{j-1} \setminus V'$ and the third follows because $v_j \in V_{j} \setminus (V' \cup \Vf)$ and each of the $\an-1$ colours in $\Cn \setminus \{c'\}$ is assigned by $\varphi_{j-1}$ to at least one vertex in $V_{j-1} \setminus (V' \cup \Vf)$.

Let $\Phi = \sum_{x \in V_j}|\A_x| -k(a+k-1)$. We will show that $\Phi>0$ and hence obtain a contradiction to the hypothesis of the lemma that $\A$ contains at most $a-k+1$ sets each of size at most $k$. We do this in two cases according to the value of $\an$.

\textbf{Case 1:} Suppose that $\an \leq k-1$. Observe that, for each $x \in V_j$, we have $|\A_x| \geq 1$  because $v_j$ is adjacent in $H_j$ to a vertex of colour $c'$ and thus $\deg_{H_j}(x) \geq \deg_{H_j}(v_j) \geq 1$ by the properties of the degeneracy ordering. So we have $\sum_{x \in V_j \setminus V'}|\A_x| \geq |V_j \setminus V'| \geq (k-1)(a-\an)+\an$ by \eqref{E:verticesLowerBounds}. Furthermore, each of the $|\Vf| +1$ vertices in $\Vf \cup \{v_j\}$ is in a set in $\A$ that also contains a vertex in $V'$ using our assumption that the claim fails and the fact that $c' \in \Cn$. Thus, because $|A|\leq k$ for each $A \in \A$, we have $\sum_{x \in V'}|\A_x| \geq \lceil \frac{1}{k-1}(|\Vf| +1) \rceil = a-\an+1$ where the equality follows by \eqref{E:verticesLowerBounds}. Using these lower bounds on $\sum_{x \in V_j \setminus V'}|\A_x|$ and $\sum_{x \in V'}|\A_x|$,
\[\Phi \geq (k-1)(a-\an)+a+1-k(a-k+1)=(k-1)(k-\an)+1.\]
Thus, since $\an \leq k-1$ by the conditions of this case,  $\Phi > 0$ and we have the required contradiction.

\textbf{Case 2:} Suppose that $\an \geq k$. We show this case cannot arise by obtaining a contradiction without the need for our assumption that the claim is false. Observe that $\deg_{H_j}(v_j) \geq \an$ by the definition of $\Cn$ and hence $\deg_{H_j}(x) \geq \an$ for each $x \in V_j$ by the properties of the degeneracy ordering. Thus we have $\deg_{H_j}(x) \leq |\A_{x}|(k-1)$ for each $x \in V_j$ and hence
\begin{equation}\label{E:hitsLowerBound}
|\A_x| \geq \tfrac{1}{k-1}\deg_{H_j}(x) \geq \tfrac{1}{k-1}\an \quad \text{for each $x \in V_j$.}
\end{equation}
So we have $\sum_{x \in V_j}|\A_x| \geq \frac{1}{k-1}\an|V_j| \geq \frac{1}{k-1}\an((k-1)(a-\an)+\an+1)$ by \eqref{E:verticesLowerBounds} and \eqref{E:hitsLowerBound}. Thus,
\begin{equation}\label{E:PhiLowerBound}
\Phi \geq \mfrac{\an((k-1)(a-\an)+\an+1)}{k-1} - k(a-k+1)=a(\an-k)+k(k-1)-\mfrac{(k-2)\an^2-\an}{k-1}.
\end{equation}
In order to show that $\Phi>0$ using \eqref{E:PhiLowerBound} we require a lower bound on $a$.

We first show that $\Cf \setminus \Cn$ is nonempty and then use this fact to obtain the required lower bound on $a$. Let $m=\max\{|A \cap V_j|:A \in \A\}$ and $A_1$ be a set in $\A$ such that $|A_1 \cap V_j|=m$. Using the definition of $m$ and a similar argument to the one used to establish \eqref{E:hitsLowerBound}, we see that $|\A_x| \geq \frac{1}{m-1}\deg_{H_j}(x) \geq \frac{1}{m-1}\an$ for each $x \in V_j$. So each vertex in $A_1 \cap V_j$ is in at least $\frac{1}{m-1}\an-1$ sets in $\A \setminus \{A_1\}$. Further, no set in $\A \setminus \{A_1\}$ can contain more than one vertex in $A_1 \cap V_j$. Thus $|\A|-1 \geq m(\frac{1}{m-1}\an-1)$ and hence, using $|\A| \leq a-k+1$, we have $a \geq \frac{m}{m-1}\an-m+k$. So we have that $a > \an$ since $m \leq k$ and hence that $\Cf \setminus \Cn$ is indeed nonempty.

Let $c''$ be a colour in $\Cf \setminus \Cn$, let $V''=\varphi^{-1}_{j-1}(c'')$, and note that $|V''|=k-1$ because $c'' \in \Cf$. No set in $\A$ can contain more than one vertex in $V''$ because $\varphi_{j-1}$ is a proper colouring, and each vertex in $V''$ is in at least $\frac{1}{k-1}\an$ sets in $\A$ by \eqref{E:hitsLowerBound}. Thus $a-k+1 \geq |\A| \geq \frac{1}{k-1}\an|V''|=\an$ and hence $a \geq \an+k-1$. Substituting this into \eqref{E:PhiLowerBound} and simplifying, remembering that $\an \geq k$ by the conditions of this case, we obtain
\[\Phi \geq \mfrac{\an(\an-k+2)}{k-1}>0\]
and we have the required contradiction.
\end{proof}

We observed in Lemma~\ref{L:tightnessGenk}(b) that, for each $k \geq 6$ with $k \equiv 2 \mod{4}$, to guarantee a $K_k$-decomposition of a graph $G$ of $k$-admissible order whose complement has at most $(\frac{n-1}{k-1}-k+1)\binom{k}{2}$ edges, we require more than  simply $G$ being $K_k$-divisible (note that $\frac{1}{4}(k^2-k-2) > k-1$ for each $k \geq 6$). It is through Lemma~\ref{L:colouring} that our proof uses the stronger assumption that $G$ is the leave of a partial $(n,k,1)$-design.  The conclusion of Lemma~\ref{L:colouring} does not hold if we merely require that $G$ be a graph of order $a(k-1)$ with at most $(a-k+1)\binom{k}{2}$ edges, even if we further demand that $G$ be $K_k$-divisible. For example, for any integer $k \geq 6$ such that $k \equiv 2 \mod{4}$, if we take $a=\frac{1}{4}(k^2+3k-2)$, then the graph of order $a(k-1)$ consisting of a copy of $K_{k(k-1)/2+1}$ and  isolated vertices has exactly $(a-k+1)\binom{k}{2}$ edges and is $K_k$-divisible, but clearly does not have a proper colouring with $a$ colours.

With Lemma~\ref{L:colouring} in hand we are now in a position to prove Theorem~\ref{T:evansBigGenk}. We find the required $K_k$-decomposition of the leave $G$ of the partial design by first applying Lemma~\ref{L:colouring} to obtain the copies of $K_k$ containing a particular vertex of minimum degree in $G$, and then using Lemma~\ref{L:mutualNeighbourhoodDecomp} to obtain the rest of the decomposition.

\begin{proof}[\textup{\textbf{Proof of Theorem~\ref{T:evansBigGenk}}}]
The second part of the theorem was proved as Lemma~\ref{L:tightnessGenk}(a), so it remains to prove the first part. Let $(V, \A)$ be a partial $(n,k,1)$-design such that $n$ is $k$-admissible and $|\A| \leq \frac{n-1}{k-1}-k+1$. Throughout the proof we assume that $n$ is large relative to $k$ and employ asymptotic notation with respect to this regime. Let $G$ be the leave of $(V, \A)$ and note that $G$ is $K_k$-divisible because $n$ is $k$-admissible. Let $z$ be a point such that $|\A_z| \geq |\A_x|$ for each $x \in V$ and let $\A'=\A \setminus \A_z$. Let $a$ be the  integer such that $|\A_z|=\frac{n-1}{k-1}-a$, and note that $a \geq k-1$ and $|\A'| \leq a-k+1$.

Let $U=N_G(z)$ and observe that $|A| = k$ for each $A \in \A'$, $|A \cap A'| \leq 1$ for all distinct $A,A' \in \A'$ and $\Gc[U] = \bigcup_{A \in \A'}K_{A \cap U}$. Thus, since $|U| = \deg_G(z) = a(k-1)$, we can apply Lemma~\ref{L:colouring} to show there is a proper colouring of $\Gc[U]$ with $a$ colours in which each colour class has order $k-1$. Thus, there is a partition $\mathcal{U}$ of $U$ such that $|\mathcal{U}|=a$ and $G[X]$ is a copy of $K_{k-1}$ for each $X \in \mathcal{U}$. Let $\B=\{X \cup \{z\}:X \in \mathcal{U}\}$.

Let $G'$ be the graph obtained from $G$ by deleting the edges in $\bigcup_{B \in \B}E(K_B)$ and the vertex $z$. It suffices to show that we can apply Lemma~\ref{L:mutualNeighbourhoodDecomp} to find a $K_k$-decomposition $\D'$ of $G'$, because then to complete $(V,\A)$ we can add the blocks in $\B$ along with blocks corresponding to the copies of $K_k$ in $\D'$. So it remains to show that $G'$ satisfies the hypotheses of Lemma \ref{L:mutualNeighbourhoodDecomp}. Since $G$ is $K_k$-divisible, so is $G'$. Observe that
\[G'= K_{V \setminus \{z\}}-\medop\bigcup_{A \in \A_z \cup \B} K_{A \setminus \{z\}} - \medop\bigcup_{A \in \A'} K_A,\]
and that each element of $V \setminus \{z\}$ is in exactly one set in $\{A\setminus\{z\}: A \in \A_z \cup \B\}$. Thus, for each $x \in V \setminus \{z\}$,
\begin{equation}\label{E:degTh1Proof}
\deg_{\overline{G'}}(x) = (k-1)|\A'_x|+k-2.
\end{equation}

Now
\begin{equation}\label{E:edgesTh1Proof}
|E(G')| = \mbinom{n}{2}- (|\A|+|\B|)\mbinom{k}{2} > \mbinom{n}{2}- k(n-1) =\mbinom{n}{2}-O(n)
\end{equation}
where the first inequality follows because $|\A| < \frac{n-1}{k-1}$ by supposition and $|\B| \leq \frac{n-1}{k-1}$ by definition. Now let $uv$ be an arbitrary edge of $G'$ and note that this implies $|\A'_{u} \cap \A'_{v}| =  0$. We have $|\A'_{u}|+|\A'_{v}| \leq \frac{2}{3}|\A|$ because $|\A'_{u}|,|\A'_{v}| \leq |\A_z|$ by the definition of $z$ and $|\A'_{u}|+|\A'_{v}| \leq |\A|-|\A_z|$. Then, using \eqref{E:degTh1Proof},
\begin{equation}\label{E:neighbourhoodTh1Proof}
|N_{G'}(u,v)| \geq n-1-(k-1)(|\A'_{u}|+|\A'_{v}|)-2(k-2) \geq \tfrac{1}{3}n-O(1)
\end{equation}
where the second inequality follows because $|\A'_{u}|+|\A'_{v}| \leq \frac{2}{3}|\A| < \frac{2(n-1)}{3(k-1)}$. In view of \eqref{E:edgesTh1Proof} and \eqref{E:neighbourhoodTh1Proof}, we can apply Lemma~\ref{L:mutualNeighbourhoodDecomp}, choosing $\gamma < \min\{1-\delta_{K_k},\frac{1}{3k}\}$, to find a $K_k$-decomposition $\D'$ of $G'$ and hence complete the proof.
\end{proof}

\section{Proof of Theorems~\ref{T:evansForAdmissibleOrder} and~\ref{T:evansForAnyGraph}}

The proofs of Theorems~\ref{T:evansForAdmissibleOrder} and~\ref{T:evansForAnyGraph} proceed along similar lines to the proof of Theorem~\ref{T:evansBigGenk}, although the details vary significantly. In each case, we first require a lemma analogous to Lemma~\ref{L:colouring}: this is Lemma~\ref{L:colouring2} in the case of Theorem~\ref{T:evansForAdmissibleOrder} and Lemma~\ref{L:colouring3} in the case of Theorem~\ref{T:evansForAnyGraph}. Like Lemma~\ref{L:colouring}, these lemmas are proved by colouring with a greedy algorithm that may recolour already-coloured vertices when required.

\begin{lemma}\label{L:colouring2}
Let $k$ and $a$ be  integers such that $k \geq 3$ and $a > \ell$, where $\ell=\frac{1}{4}(k^2-k-2)$.
Let $H$ be a graph of order $a(k-1)$ such that
$\sum_{x \in V(H)}\lceil\frac{1}{k-1}\deg_H(x)\rceil < k(a-\ell)$. Then $H$ has a proper colouring with $a$ colours such that each colour class contains $k-1$ vertices.
\end{lemma}

\begin{proof}
Note that $\ell$ may not be an integer, but $2\ell=\binom{k}{2}-1$ is an integer. The set-up of the proof proceeds identically to that of the proof of Lemma~\ref{L:colouring} up to and including the paragraph after the claim. So we adopt all the notation defined up to that point and see that it suffices to prove the claim there, which we restate below.
\begin{claim}
There is a vertex in $\Vf$ that is not adjacent in $H_j$ to any vertex in $V'$.
\end{claim}

\noindent \textbf{Proof of claim.} Recall that $v_1,\ldots,v_{a(k-1)}$ is a degeneracy ordering of $V(H)$, $V_i=\{v_1,\ldots,v_i\}$ and $H_i=H[V_i]$ for each $i \in \{1,\ldots a(k-1)\}$ and $\varphi_{j-1}$ is a legal colouring of $H_{j-1}$ with a set $C$ of $a$ colours for some $j \in \{a+1,\ldots,a(k-1)\}$. Further, $V'=\varphi^{-1}_{j-1}(c')$ and $\Vf=\bigcup_{c\in \Cf \setminus \Cn}\varphi^{-1}_{j-1}(c)$ where $c'$ is a colour in $\Cn \setminus \Cf$, $\Cf=\{c \in C: |\varphi_{j-1}^{-1}(c)|=k-1\}$ and $\Cn$ is the set of $\an \geq 1$ colours in $C$ that are assigned by $\varphi_{j-1}$ to vertices adjacent in $H_j$ to $v_j$.

Suppose for a contradiction that each vertex in $\Vf$ is adjacent in $H_j$ to some vertex in $V'$. As in the proof of Lemma~\ref{L:colouring}, observe that $V'$ and $\Vf$ are disjoint and that
\begin{equation}\label{E:verticesLowerBounds2}
|V'| \geq 1, \qquad |\Vf|=(k-1)(a-\an) \qquad \text{and} \qquad |V_j \setminus (V' \cup \Vf)| \geq \an.
\end{equation}

Let $r_x = \lceil\frac{1}{k-1}\deg_{H_j}(x)\rceil$ for each $x \in V$ and let $\Phi=\sum_{x \in V_j}r_x-k(a-\ell)$. We will complete the proof by showing that $\Phi \geq 0$ and hence obtaining a contradiction to the hypothesis of the lemma that $\sum_{x \in V(H)}\lceil\frac{1}{k-1}\deg_H(x)\rceil < k(a-\ell)$. We do this in two cases according to the value of $\an$.

\textbf{Case 1:} Suppose that $\an \leq k-1$. Observe that, for each $x \in V_j$, we have $r_{x} \geq 1$ for all $x \in V_j$  because $v_j$ is adjacent in $H_j$ to a vertex of colour $c'$ and thus $\deg_{H_j}(x) \geq \deg_{H_j}(v_j) \geq 1$ by the properties of the degeneracy ordering. So we have $\sum_{x \in V_j \setminus V'}r_x \geq |V_j \setminus V'| \geq (k-1)(a-\an)+\an$ by \eqref{E:verticesLowerBounds2}. Furthermore, each of the $|\Vf| +1$ vertices in $\Vf \cup \{v_j\}$ is adjacent in $H_j$ to a vertex in $V'$ using our assumption that the claim fails and the fact that $c' \in \Cn$. Thus, $\sum_{x \in V'}\deg_{H_j}(x) \geq |\Vf| +1$ and so $\sum_{x \in V'}r_x \geq \lceil \frac{1}{k-1}(|\Vf| +1) \rceil = a-\an+1$ where the equality follows by \eqref{E:verticesLowerBounds2}. Using these lower bounds on $\sum_{x \in V \setminus V'}r_x$ and $\sum_{x \in V'}r_x$,
\[\Phi \geq (k-1)(a-\an)+a+1-k(a-\ell)=k \ell-\an(k-1)+1 \geq k(\ell-k+2),\]
where the last inequality follows by using the condition of this case that $\an \leq k-1$ and simplifying. Thus $\Phi \geq 0$ and we have the required contradiction because it is easily checked that $\ell \geq k-2$ since $k \geq 3$.

\textbf{Case 2:} Suppose that $\an \geq k$. We show this case cannot arise by obtaining a contradiction without the need for our assumption that the claim is false. Observe that $\deg_{H_j}(v_j) \geq \an$ by the definition of $\Cn$ and hence $\deg_{H_j}(x) \geq \an$ for each $x \in V_j$ by the properties of the degeneracy ordering. Thus,
\begin{equation}\label{E:hitsLowerBound2}
r_x \geq \tfrac{1}{k-1}\an \quad \text{for each $x \in V_j$.}
\end{equation}
So we have $\sum_{x \in V_j}r_x \geq \frac{1}{k-1}\an|V_j| \geq \frac{1}{k-1}\an((k-1)(a-\an)+\an+1)$  by \eqref{E:verticesLowerBounds2} and \eqref{E:hitsLowerBound2}. Thus,
\[\Phi \geq \mfrac{\an((k-1)(a-\an)+\an+1)}{k-1} - k(a-\ell)=a(\an-k)+k\ell-\mfrac{(k-2)\an^2-\an}{k-1} \geq k \ell-\mfrac{\an(k^2-k-1-\an)}{k-1} \]
where for the last inequality we substituted $a \geq \an$ in view of the condition of this case that $\an \geq k$. It is routine to check that $\an(k^2-k-1-\an) \leq k(k-1) \ell$ using the definition of $\ell$ and the fact that either $\an \leq \binom{k}{2}-1$ or $\an \geq \binom{k}{2}$ since $\an$ is an integer. Thus $\Phi \geq 0$ and we have the required contradiction.
\end{proof}

As suggested by the proof of Lemma~\ref{L:tightnessGenk}(b), for any $k \equiv 2 \mod{4}$, the tightness of Lemma~\ref{L:colouring2} can be seen by taking $a=\frac{1}{2}k(k-1)$ and considering the graph of order $a(k-1)$  consisting of a copy of $K_{a+1}$ and isolated vertices.

\begin{proof}[\textbf{\textup{Proof of Theorem~\ref{T:evansForAdmissibleOrder}}}]
The second part of the theorem follows by Lemma~\ref{L:tightnessGenk}(b), so it remains to prove the first part. Let $G$ be a $K_k$-divisible graph of order $n$ such that $n \equiv 1 \mod{(k-1)}$ and $|E(\overline{G})| < (\tfrac{n-1}{k-1}-\ell)\tbinom{k}{2}$. Throughout the proof we assume that $n$ is large relative to $k$.

Observe that $\deg_{\Gc}(x) \equiv 0 \mod{k-1}$ for each $x \in V(G)$ since $G$ is $K_k$-divisible and $n \equiv 1 \mod{(k-1)}$. Let $z$ be a vertex of minimum degree in $G$ and let $U=N_G(z)$. Since $G$ is $K_k$-divisible there is an integer $a$ such that $|U|=\deg_{G}(z)=a(k-1)$.  Now $\deg_{\Gc}(z)=n-1-a(k-1)$, and each of the $n-1-a(k-1)$ vertices in $N_{\Gc}(z)$ has positive degree in $\Gc$ and hence has degree at least $k-1$. Thus $\sum_{x \in V(G) \setminus U}\deg_{\Gc}(x) \geq k(n-1-a(k-1))$, so
\[k(n-1-a(k-1))+\medop\sum_{x \in U}\deg_{\Gc}(x) \leq 2|E(\Gc)|< k(k-1)\left(\tfrac{n-1}{k-1}-\ell\right)\]
and hence $\sum_{x \in U}\deg_{\Gc}(x) < k(k-1)(a-\ell)$.  Thus, again using $\deg_{\Gc}(x) \equiv 0 \mod{k-1}$ for each $x \in V(G)$,
\[\medop\sum_{x \in U}\lceil\tfrac{1}{k-1}\deg_{\Gc[U]}(x)\rceil  \leq \medop\sum_{x \in U}\lceil\tfrac{1}{k-1}\deg_{\Gc}(x)\rceil = \medop\sum_{x \in U}\tfrac{1}{k-1}\deg_{\Gc}(x) < k(a-\ell).\]
So we can apply Lemma~\ref{L:colouring2} to find a proper colouring of  $\Gc[U]$ with $a$ colours in which each colour class has order $k-1$. Thus, there is a partition $\mathcal{U}$ of $U$ such that $|\mathcal{U}|=a$ and $G[X]$ is a copy of $K_{k-1}$ for each $X \in \mathcal{U}$. Let $\D=\{K_{X \cup \{z\}}:X \in \mathcal{U}\}$.

Let $G'$ be the graph obtained from $G$ by removing the edges of each copy of $K_k$ in $\D$ and then deleting the (now isolated) vertex $z$. It suffices to show that we can apply Lemma~\ref{L:mutualNeighbourhoodDecomp} to find a $K_k$-decomposition $\D'$ of $G'$, for then $\D \cup \D'$ will be a $K_k$-decomposition of $G$. Since $G$ is $K_k$-divisible, so is $G'$. Now,
\begin{equation}\label{E:edgesTh2Proof}
|E(G')| = \mbinom{n}{2}-|E(\Gc)|-|\D|\mbinom{k}{2} > \mbinom{n}{2}-k(n-1) =\mbinom{n}{2}-O(n)
\end{equation}
where the first inequality follows because $|E(\Gc)| < \frac{n-1}{k-1}\binom{k}{2}$ and $|\D| \leq \frac{n-1}{k-1}$. Let $uv$ be an arbitrary edge of $G'$, let $T=(N_{\Gc}(u) \cup N_{\Gc}(v)) \setminus \{z\}$, and note that $u,v \notin T$. Each vertex in $T$ has positive degree in $\Gc$ and hence degree at least $k-1$. Also $\deg_{\Gc}(u)+\deg_{\Gc}(v) \geq |T|$ and hence $\deg_{\Gc}(z) \geq \frac{1}{2}|T|$ by the definition of $z$. Thus we have
\[\tfrac{3}{2}|T|+(k-1)|T| \leq \medop\sum_{x \in \{u,v,z\}}\deg_{\Gc}(x)+\medop\sum_{x \in T}\deg_{\Gc}(x) \leq 2|E(\Gc)| < kn\]
and hence $|T| \leq \frac{2k}{2k+1}n$. So we have $|T'| \leq \frac{2k}{2k+1}n+O(1)$, where $T'=N_{\overline{G'}}(u) \cup N_{\overline{G'}}(v)$, because it follows from the definition of $G'$ that $T'$ can be obtained from $T$ by adding at most $2(k-1)$ vertices. Thus $|N_{G'}(u,v)| = n-3-|T'| \geq \frac{1}{2k+1}n-O(1)$. By this fact and \eqref{E:edgesTh2Proof},  we can apply Lemma~\ref{L:mutualNeighbourhoodDecomp}, choosing $\gamma< \min\{1-\delta_{K_k},\frac{1}{k(2k+1)}\}$, to find a $K_k$-decomposition $\D'$ of $G'$ and hence complete the proof.
\end{proof}

In Lemma~\ref{L:colouring3}, we are forced to prove a slightly stronger result for $k=3$ so as to eventually obtain a tight result for $k=3$ in Theorem~\ref{T:evansForAnyGraph}.

\begin{lemma}\label{L:colouring3}
Let $k$ and $a$ be  integers such that $k \geq 3$ and $a \geq 1$. Let $H$ be a graph of order $a(k-1)$ such that either
\begin{itemize}
    \item[\textup{(i)}]
$\sum_{x \in V(H)}\lceil\frac{1}{k-1}(\deg_H(x)-1)\rceil \leq a-2$; or
    \item[\textup{(ii)}]
$k=3$, $\Delta(H) \leq 2a-2$, and $\sum_{x \in V(H)}\lceil\frac{1}{k-1}(\deg_H(x)-1)\rceil \leq a$.
\end{itemize}
Then $H$ has a proper colouring with $a$ colours such that each colour class has order $k-1$.
\end{lemma}

\begin{proof}
The set-up of the proof proceeds identically to that of the proof of Lemma~\ref{L:colouring} up to and including the paragraph after the claim. So we adopt all the notation defined up to that point and see that it suffices to prove the claim there, which we restate below.
\begin{claim}
There is a vertex in $\Vf$ that is not adjacent in $H_j$ to any vertex in $V'$.
\end{claim}

\noindent \textbf{Proof of claim.} Recall that $v_1,\ldots,v_{a(k-1)}$ is a degeneracy ordering of $V(H)$, $V_i=\{v_1,\ldots,v_i\}$ and $H_i=H[V_i]$ for each $i \in \{1,\ldots a(k-1)\}$ and $\varphi_{j-1}$ is a legal colouring of $H_{j-1}$ with a set $C$ of $a$ colours for some $j \in \{a+1,\ldots,a(k-1)\}$. Further, $V'=\varphi^{-1}_{j-1}(c')$ and $\Vf=\bigcup_{c\in \Cf \setminus \Cn}\varphi^{-1}_{j-1}(c)$ where $c'$ is a colour in $\Cn \setminus \Cf$, $\Cf=\{c \in C: |\varphi_{j-1}^{-1}(c)|=k-1\}$ and $\Cn$ is the set of $\an \geq 1$ colours in $C$ that are assigned by $\varphi_{j-1}$ to vertices adjacent in $H_j$ to $v_j$.

Suppose for a contradiction that each vertex in $\Vf$ is adjacent in $H_j$ to some vertex in $V'$. As in the proof of Lemma~\ref{L:colouring}, and further noting that $|V'| \leq k-2$ because $c' \notin \Cf$, observe that $V'$ and $\Vf$ are disjoint and that
\begin{equation}\label{E:verticesLowerBounds3}
k-2 \geq |V'| \geq 1 \qquad \text{and} \qquad |\Vf|=(k-1)(a-\an).
\end{equation}
We consider two cases based on whether $\an=1$.

\textbf{Case 1:} Suppose $\an=1$. Then, since $c' \in \Cn \setminus \Cf$, it must be the case that $\Cn = \Cn \setminus \Cf=\{c'\}$. It follows that $\Cf=C\setminus\{c'\}$ because $\Cf \cup \Cn=C$. Now each vertex in $\Vf$ is adjacent in $H_j$ to a vertex in $V'$ using our assumption that the claim fails. Thus $\sum_{x \in V'}\deg_{H_j}(x) \geq |\Vf|$ and so
\[\medop\sum_{x \in V'}\tfrac{1}{k-1}(\deg_{H_j}(x)-1) \geq \tfrac{1}{k-1}|\Vf| - \tfrac{1}{k-1}|V'| > a-2\]
where the last inequality follows because $|V'| \leq k-2$ and $\Vf=(k-1)(a-1)$ by \eqref{E:verticesLowerBounds3} since $\an=1$.  This contradicts (i) of our hypotheses, so we may assume that (ii) holds and hence $k=3$ and $\Delta(H) \leq 2a-2$. Then $|V'|=\{y\}$ for some $y \in V_{j-1}$ because $1 \leq |V'|\leq k-2=1$ by \eqref{E:verticesLowerBounds3}. Thus $y$ is adjacent in $H_j$ to each of the $(k-1)(a-1)=2a-2$ vertices in $\Vf$ by our assumption that the claim fails. Furthermore, $y$ is adjacent in $H_j$ to $v_j$ since $c' \in \Cn$. Thus $\deg_{H_j}(y) \geq 2a-1$ in contradiction to our assumption that $\Delta(H) \leq 2a-2$.

\textbf{Case 2:} Suppose $\an \geq 2$. We show this case cannot arise by obtaining a contradiction without the need for our assumption that the claim is false. Then $\deg_{H_j}(v_j) \geq \an \geq 2$ by the definition of $\Cn$. So, for each $x \in V_j$, we have $\deg_{H_j}(x) \geq 2$ by the properties of the degeneracy ordering and hence $\lceil\frac{1}{k-1}(\deg_{H_j}(x)-1)\rceil \geq 1$. But then we have $\sum_{x \in V_j}\lceil\frac{1}{k-1}(\deg_{H_j}(x)-1)\rceil \geq j$ which contradicts both (i) and (ii) of our hypotheses since $j \geq a+1$.
\end{proof}

For each odd $k \geq 5$ and each $a \geq 2$, the tightness of the condition $\sum_{x \in V(H)}\lceil\frac{1}{k-1}(\deg_H(x)-1)\rceil \leq a-2$ in Lemma~\ref{L:colouring3} is witnessed by the graph of order $a(k-1)$ that is the vertex disjoint union of a star with $(a-1)(k-1)+1$ edges and a perfect matching with $\frac{1}{2}(k-3)$ edges. In any proper colouring of such a graph, the colour assigned to the centre vertex of the star must be assigned to fewer than $k-1$ vertices. The proof of Theorem~\ref{T:evansForAnyGraph} differs from the proof of Theorems~\ref{T:evansBigGenk} and~\ref{T:evansForAdmissibleOrder} in that it appears that the order, and hence the degrees, of $G$ can belong to any congruence class modulo $k-1$. However we quickly see that the critical case is when the order of $G$ is congruent to 2 modulo $k-1$.

\begin{proof}[\textbf{\textup{Proof of Theorem~\ref{T:evansForAnyGraph}}}]
The second part of the theorem follows by Lemma~\ref{L:tightnessGenk}(c) and Lemma~\ref{L:k=3tightness}(a), so it remains to prove the first part. Let $G$ be a $K_k$-divisible graph of order $n$ such that either $|E(\Gc)| < n-\frac{1}{2}(k+1)$ or $k=3$ and $|E(\Gc)| < n$. Then, because $G$ cannot be $K_3$-divisible if $|E(\Gc)| = n-2$, in fact we have either
\begin{itemize}[nosep]
    \item
$|E(\Gc)| < n-\frac{1}{2}(k+1)$; or
    \item
$k=3$ and $|E(\Gc)| = n-1$.
\end{itemize}
We assume that $n$ is large relative to $k$ and consider three cases according to the congruence class of $n$ modulo $k-1$.

\textbf{Case 1:} Suppose that $k \geq 4$ and $n-1 \equiv j \mod{(k-1)}$ for some $j \in \{2, \ldots, k-2\}$. Then, because $G$ is $K_k$-divisible, $\deg_{\Gc}(x) \equiv j \mod{(k-1)}$ for each $x \in V(G)$.
Therefore, $|E(\Gc)| \geq \frac{1}{2}jn \geq n$, contradicting our assumption. So this case cannot arise.

\textbf{Case 2:} Suppose that $n-1 \equiv 0 \mod{(k-1)}$. Then, because $G$ is $K_k$-divisible, $\deg_{\Gc}(x) \equiv 0 \mod{(k-1)}$ for each $x \in V(G)$. Let $uv$ be an arbitrary edge of $G$. Let $T=N_{\overline{G}}(u) \cup N_{\overline{G}}(v)$ and note that $u,v \notin T$ and $|T| \leq \deg_{\overline{G}}(u)+\deg_{\overline{G}}(v)$. Also, $\deg_{\overline{G}}(x)$ is positive for each $x \in T$ and hence at least $k-1$. We have
\[|T|+(k-1)|T| \leq \medop\sum_{x \in \{u,v\}}\deg_{\Gc}(x)+\medop\sum_{x \in T} \deg_{\Gc}(x) \leq 2|E(\Gc)| < 2n.\]
So $|T| < \frac{2n}{k} \leq \frac{2n}{3}$ since $k \geq 3$. Therefore, $|N_G(u,v)| = n-2-|T| \geq \frac{n}{3} -O(1)$.
We also have $|E(G)| > \binom{n}{2}-n$.  In view of these two facts,  we can apply Lemma~\ref{L:mutualNeighbourhoodDecomp}, choosing $\gamma<\min\{1-\delta_{K_k}, \frac{1}{3k}\}$, to find a $K_k$-decomposition $\D$ of $G$ and hence complete the proof.

\textbf{Case 3:} Suppose that $n-1 \equiv 1 \mod{(k-1)}$. Then, because $G$ is $K_k$-divisible, $\deg_{\Gc}(x) \equiv 1 \mod{(k-1)}$ for each $x \in V(G)$. It will be convenient to define $\rho=0$ if $|E(\Gc)| < n-\frac{1}{2}(k+1)$ and $\rho=2$ if $k=3$ and $|E(\Gc)| = n-1$, so that we always have
$|E(\Gc)| < n-\frac{1}{2}(k+1)+\rho$.

Let $z$ be a vertex of minimum degree in $G$ and let $U=N_G(z)$. We will show that $\Gc[U]$ obeys the hypotheses of Lemma~\ref{L:colouring3}. Since $G$ is $K_k$-divisible there is an integer $a$ such that $|U|=\deg_{G}(z)=a(k-1)$. We may assume that $a \geq 1$ for otherwise $a=0$, $k=3$, $\Gc$ is a star with $n-1$ edges and hence $G$ is $K_3$-decomposable because its edges form a copy of $K_{n-1}$ and $G$ is $K_3$-divisible by assumption. Now $\deg_{\Gc}(z)=n-1-a(k-1)$, and each of the $n-1-a(k-1)$ vertices in $N_{\Gc}(z)$ has degree at least 1 in $\Gc$. Thus $\sum_{x \in V(G) \setminus U}\deg_{\Gc}(x) \geq 2n-2-2a(k-1)$, so
\begin{equation}\label{E:Th3degSuminU}
\medop\sum_{x \in U}\deg_{\Gc}(x) \leq 2|E(\Gc)|-(2n-2-2a(k-1)) < (2a-1)(k-1)+2\rho
\end{equation}
where the last inequality follows because $|E(\Gc)| < n-\frac{1}{2}(k+1)+\rho$. Thus,
\[\medop\sum_{x \in U}\lceil\tfrac{1}{k-1}(\deg_{\Gc[U]}(x)-1)\rceil  \leq \medop\sum_{x \in U}\lceil\tfrac{1}{k-1}(\deg_{\Gc}(x)-1)\rceil = \tfrac{1}{k-1}\medop\sum_{x \in U}\deg_{\Gc}(x)-a < a-1+\rho\]
where the equality holds because $|U|=a(k-1)$ and $\deg_{\Gc}(x) \equiv 1 \mod{k-1}$ for each $x \in U$, and the last inequality follows using \eqref{E:Th3degSuminU} and the fact that $\frac{2}{k-1}\rho=\rho$ in all cases. So we in fact have $\sum_{x \in U}\lceil\tfrac{1}{k-1}(\deg_{\Gc[U]}(x)-1)\rceil \leq a-2+\rho$ because the terms are all integers.
So if $\rho=0$, then $H$ obeys (i) in the hypotheses of Lemma~\ref{L:colouring3}. If $\rho=2$ and hence $k=3$ and $|E(\Gc)|=n-1$, then $\Delta(\Gc[U]) \leq 2a-2$ for otherwise $\Gc$ would have to be a graph obtained by adding exactly one edge to the vertex disjoint union of a star with $n-2a-1$ edges and a star with $2a-1$ edges. This contradicts the fact that each vertex of $\Gc$ has odd degree. So if $\rho=2$, then $H$ obeys (ii) in the hypotheses of Lemma~\ref{L:colouring3}. Thus, by Lemma~\ref{L:colouring3} there exists a proper colouring of $\Gc[U]$ with $a$ colours in which each colour class has order $k-1$. So there is a partition $\mathcal{U}$ of $U$ such that $|\mathcal{U}|=a$ and $G[X]$ is a copy of $K_{k-1}$ for each $X \in \mathcal{U}$. Let $\D=\{K_{X \cup \{z\}}:X \in \mathcal{U}\}$.

Let $G'$ be the graph obtained from $G$ by removing the edges of each copy of $K_k$ in $\D$ and then deleting the (now isolated) vertex $z$. It suffices to show that we can apply Lemma~\ref{L:mutualNeighbourhoodDecomp} to find a $K_k$-decomposition $\D'$ of $G'$, for then $\D \cup \D'$ will be a $K_k$-decomposition of $G$. Since $G$ is $K_k$-divisible, so is $G'$.

Let $uv$ be an arbitrary edge of $G'$, let $T=(N_{\Gc}(u) \cup N_{\Gc}(v)) \setminus \{z\}$, and note that $u,v \notin T$. Furthermore $\deg_{\Gc}(u)+\deg_{\Gc}(v) \geq |T|$. At most two edges of $\Gc$ are incident with two vertices in $\{u,v,z\}$ and hence
\[\deg_{\overline{G}}(u) + \deg_{\overline{G}}(v) + \deg_{\overline{G}}(z) \leq |E(\overline{G})| + 2 \leq n+1.\]
Thus, because $\deg_{\overline{G}}(u),\deg_{\overline{G}}(v)  \leq \deg_{\overline{G}}(z)$ by the definition of $z$, we have that $|T| \leq \deg_{\overline{G}}(u)+\deg_{\overline{G}}(v) \leq \frac{2}{3}n+O(1)$. So, considering the way in which $G'$ is obtained from $G$,  $N_{G'}(u,v) \geq n-3 - |T| - 2(k-1) > \frac{1}{3}n-O(1)$. We also have
\[|E(G')| = \mbinom{n}{2}-|E(\Gc)|-|\D|\mbinom{k}{2} > \mbinom{n}{2}-n-\tfrac{1}{2}k(n-1) =\mbinom{n}{2}-O(n)\]
because $|E(\Gc)| < n$ and $|\D| \leq \frac{n-1}{k-1}$. In view of these two facts,  we can apply Lemma~\ref{L:mutualNeighbourhoodDecomp}, choosing $\gamma< \min\{1-\delta_{K_k}, \frac{1}{3k}\}$, to find a $K_k$-decomposition $\D'$ of $G'$ and so complete the proof.
\end{proof}

The proof of Corollary~\ref{C:k=3} follows easily from Theorems~\ref{T:evansForAdmissibleOrder} and~\ref{T:evansForAnyGraph} and Lemma~\ref{L:k=3tightness}.

\begin{proof}[\textbf{\textup{Proof of Corollary~\ref{C:k=3}}}]
For $n \equiv 1,3 \mod{6}$ the result follows immediately from Theorem~\ref{T:evansForAdmissibleOrder} and for $n \equiv 0 \mod{6}$ the result follows immediately from Theorem~\ref{T:evansForAnyGraph}. For $n \equiv 5 \mod{6}$, Lemma~\ref{L:k=3tightness}(b) gives a $K_3$-divisible graph with $\binom{n}{2}-\frac{1}{2}(3n-7)$ edges that has no $K_3$-decomposition and, furthermore, any $K_3$-divisible graph of order $n$ with more than $\binom{n}{2}-\frac{1}{2}(3n-7)$ edges has at least $\binom{n}{2}-\frac{1}{2}(3n-13)$ edges and hence is $K_3$-decomposable by Theorem~\ref{T:evansForAdmissibleOrder} if $n$ is sufficiently large. For $n \equiv 2,4 \mod{6}$, Lemma~\ref{L:k=3tightness}(c) gives a $K_3$-divisible graph with $\binom{n}{2}-n-2$ edges that has no $K_3$-decomposition and, furthermore, any $K_3$-divisible graph of order $n$ with more than $\binom{n}{2}-n-2$ edges has at least $\binom{n}{2}-n+1$ edges and hence is $K_3$-decomposable by Theorem~\ref{T:evansForAnyGraph} if $n$ is sufficiently large.
\end{proof}

\section{Conclusion}

The work here leaves many avenues for further investigation. It would of course be desirable to establish results similar to ours for all $n$ rather than simply for sufficiently large $n$. However, for general $k$, even the existence problem for $(n,k,1)$-designs is only resolved for large $n$. Even for values of $k$ where the existence problem is completely solved, such an improvement of our results would not be achievable using the techniques we have employed here, due to their reliance on the decomposition results in \cite{GlockKuhnLoMontgomeryOsthus2019}. One could also ask for results similar to Theorem~\ref{T:evansBigGenk} for partial $(n,k,\lambda)$-designs for $\lambda \geq 2$. It may be that the techniques used here could be adapted to prove such results. As mentioned in Section~\ref{S:prelim}, Theorems~\ref{T:evansForAdmissibleOrder} and~\ref{T:evansForAnyGraph} are not necessarily tight for all $k$, and so there is the possibility of improving them for specific values of $k$. Further, one could attempt to prove results analogous to Corollary~\ref{C:k=3} for values of $k$ other than $3$. These last two possible goals may involve significant case analysis, however. Finally, Lemma~\ref{L:mutualNeighbourhoodDecomp} suggests the problem of investigating what conditions on the size and number of triangles per edge of a graph are sufficient to guarantee that it has a $K_k$-decomposition.

\bigskip
\noindent\textbf{Acknowledgments.}
The second author was supported by Australian Research Council grants DP150100506 and FT160100048.



\end{document}